\documentclass[a4paper,openright,twoside]{article}

\usepackage[margin=1in]{geometry}

\usepackage{mathtools}
\usepackage[colorlinks=true, linkcolor=black, citecolor=black, urlcolor=black]{hyperref}
\usepackage{amsmath}
\usepackage{amsthm}
\usepackage[alphabetic, nobysame]{amsrefs}
\usepackage{amssymb}
\usepackage{graphicx}
\usepackage{eucal}
\usepackage{dsfont}
\usepackage{enumerate}
\usepackage{cancel}
\usepackage[all]{xy}

\usepackage{amsrefs}

\usepackage{longtable}

\usepackage{tensor} 

\usepackage{tikz}
\usetikzlibrary{shapes,arrows}
\usetikzlibrary{calc}

\def\RR{\mathds R}

\DeclareMathOperator{\Ima}{Im}

\DeclareMathOperator{\Gaug}{Gauge}

\DeclareMathOperator{\Lie}{Lie}

\DeclareMathOperator{\Rep}{Rep}

\DeclareMathOperator{\const}{const}

\def\pr{\textrm{pr}}

\newcommand{\G}{\mathcal{G}}

\def\({\left(}
\def\){\right)}
\def\[{\left[}
\def\]{\right]}

\def\al{\alpha}
\def\be{\beta}

\def\ga{\gamma}

\def\te{\theta}

\def\om{\omega}

\def\Om{\Omega}

\theoremstyle{plain}
\newtheorem{thm}{Proposition}[section]
\newtheorem{thmm}[thm]{Theorem}
\newtheorem{cor}[thm]{Corollary}
\newtheorem{lem}[thm]{Lemma}

\theoremstyle{definition}

\newtheorem{defnx}[thm]{Definition}
\newenvironment{defn}
  {\pushQED{\qed}\defnx}
  {\popQED\enddefnx}

\newtheorem{obsx}[thm]{Remark}
\newenvironment{obs}
  {\pushQED{\qed}\obsx}
  {\popQED\endobsx}

\newtheorem{examplex}[thm]{Example}
\newenvironment{exmp}
  {\pushQED{\qed}\examplex}
  {\popQED\endexamplex}

\def\arr{\rightarrow} 
\def\then{\Rightarrow} 

\def\h{\mathfrak{h}} 
\def\g{\mathfrak{g}} 
\def\l{\mathfrak{l}} 

\begin{document}

\title{Cartan geometries and multiplicative forms}

\author{Francesco Cattafi \footnote{Present affiliation: Erwin Schr\"odinger International Institute for Mathematics and Physics, Austria.
\newline Affiliation at submission: Department of Mathematics, KU Leuven, Belgium.
\newline Permanent e-mail address: \url{francesco.cattafi91@gmail.com} }
}

\date{}

\maketitle

\begin{abstract}
In this paper we show that Cartan geometries can be studied via transitive Lie groupoids endowed with a special kind of vector-valued multiplicative 1-forms. This viewpoint leads us to a more general notion, that of {\it Cartan bundle}, which encompasses both Cartan geometries and $G$-structures.
\end{abstract}


\begin{center}
\textbf{MSC2010}: 58A10, 53C05, 53C10, 53C15, 58H05
\end{center}

\maketitle

\section{Introduction}

The history of Cartan geometries is well known and dates back to the XIX century, when mathematicians began a sistematic study of non-euclidean geometries. In this perspective, the idea of Felix Klein was to shift the attention from the geometric objects to their symmetries: the slogan of his so-called Erlangen program was that each ``geometry" should be described by a specific group of transformations. 
Later, \'Elie Cartan took these geometries as standard models and used them to give rise to his \textit{espaces g\'en\'eralis\'es}. 




The research in this field then progressed on two paths. On the one hand, many authors used Cartan's ideas to obtain important results on relevant examples, such as parabolic geometries (see \cite{Cap09}). On the other hand, people used Cartan's approach to develop a general framework for studying these geometries; the standard modern reference is the famous book {\it Differential Geometry: Cartan's generalisation of Klein's Erlangen program} \cite{Sha97} by Richard Sharpe.

Our interest in these topics sparked from a different area in geometry. Recently, the concept of {\it Pfaffian groupoid} have been introduced \cite{Sal13} in order to understand the structure behind the jet groupoid of a Lie pseudogroup. Our original goal was to give an alternative description of the class of {\it transitive} Pfaffian groupoids, using the unique (up to isomorphism) principal bundle associated to any transitive groupoid. It has been quite an astonishing surprise to discover that, from the object we obtained, called a {\it Cartan bundle}, one could recover as a particular case the definition of a Cartan geometry. We believe that this new perspective can shed more lights in this field; for instance, one can wonder how to develop a theory of deformation for Cartan bundles, making use of our correspondence and of the known results for Lie groupoids. We are currently investigating further applications in \cite{Acc20}.

\paragraph{Main results}
Let us give a few more details on our main contributions. A Cartan geometry is defined as a principal $H$-bundle $P$ together with an equivariant vector-valued 1-form $\te$, called a Cartan connection, satisfying certain properties. We will define a {\it Cartan bundle} as a principal $H$-bundle $P$ together with an equivariant vector-valued 1-form $\te$ with conditions less restrictive than those of a Cartan geometry (see Definition \ref{definition_Cartan_bundle}). As anticipated above, we will prove a bijective correspondence between Cartan bundles and transitive Pfaffian groupoids (Theorem \ref{correspondence_cartan_bundles_pfaffian_groupoids}). In particular, when $\ker(\te) = 0$, one recovers the standard Cartan geometries, which correspond to a precise subclass of transitive Pfaffian groupoids: those whose symbol space is zero (see Theorem \ref{correspondence_pfaffian_groupoids_zero_symbol}).

Recall also that any (reductive) Cartan geometry on $M$ defines a $H$-structure, i.e.\ a reduction of the structure group of the frame bundle of $M$, together with a principal connection on it. We will show that, in the formalism of Cartan bundles $(P,\te)$, we can consider the particular case when $\ker(\te)$ coincides with the vertical bundle of $P$, and recover precisely the class of $H$-structures on $M$ without any choice of a connection. Moreover, these structures are in bijection with transitive Pfaffian groupoids whose symbol space is maximal.

\paragraph{Structure of the paper}
Section 2 is introductive, and reviews the basics on Cartan geometries and $G$-structures. Section 3 builds on the previous one and contains the main results, namely Theorem \ref{correspondence_pfaffian_groupoids_zero_symbol}, Definition \ref{definition_Cartan_bundle} and Theorem \ref{correspondence_cartan_bundles_pfaffian_groupoids}.

Through this paper we will use, without recalling the basics, the theory of Lie groupoids, Lie algebroids and principal bundles. For an introduction on these topics we refer to \cites{Mac87, Mac05, Moe03, Cra06}. Multiplicative forms on Lie groupoids are also a standard notion (see \cite{Kos16} for a nice overview) but somehow less known, especially in the case when the coefficients are not trivial. Since they constitute our main tool, we included an appendix with the definitions and the statements we use in the rest of the paper. Some of those are not simply technical lemmas but are original results, which appeared in greater generality in the author's PhD thesis \cite{Cat20}. 

\paragraph{Acknowledgements}
The author would like to thank Luca Accornero and Marius Crainic for useful discussions and comments on the first draft of this paper, as well as the anonymous referees for the improvements they suggested. The author was partially supported by the NWO under VICI project 639.033.312 (The Netherlands) and by the FWO under EOS project G0H4518N (Belgium), and is a member of the GNSAGA (INdAM).

\section{Cartan geometries}\label{section_cartan_geometries}

Let us recall the basic definitions and properties of Cartan geometries.

\begin{defn}\label{def_Klein_geometry}
A {\bf Klein geometry} is a pair $(G,H)$, where $G$ is a Lie group and $H \subseteq G$ a Lie subgroup such that the quotient manifold $G/H$ is connected. A {\bf Klein pair} is a pair $(\g,\h)$ of a Lie algebra $\g$ and a Lie subalgebra $\h \subseteq \g$. 

A {\bf model geometry} consists of a Klein pair $(\g,\h)$ together with the choice of an integration $H$ of $\h$ and of a representation $H \arr GL (\mathfrak{g})$ which extends the adjoint representation $Ad: H \arr GL (\mathfrak{h})$.
\end{defn}


\begin{defn}[Definition 3.1 of {\cite[chapter 5]{Sha97}}]\label{def_Cartan_geometry}
\index{Cartan geometry}
Let $(\mathfrak{g}, \mathfrak{h})$ be a model geometry. A {\bf Cartan geometry} $(P,\te)$ modelled on $(\mathfrak{g}, \mathfrak{h})$ is a principal $H$-bundle $P \arr M$ together with a form $\te \in \Om^1 (P, \mathfrak{g})$, called a {\bf Cartan connection} on $P$, such that
\begin{itemize}
\item $\te$ is a pointwise isomorphism, i.e.\ $\te_p: T_p P \arr \g$ is a linear isomorphism for every $p \in P$
\item $\te$ is $H$-equivariant, i.e.\ $(R_h)^* \te = h^{-1} \cdot \te$ for every $h \in H$
\item $\te( v^\dagger ) = v$ for every $v \in \h$, where $v^\dagger \in \mathfrak{X}(P)$ denotes the fundamental vector field corresponding to $v$ w.r.t.\ the $H$-action on $P$. \qedhere
\end{itemize}
\end{defn}

It follows by dimension counting that $\dim(M) = \dim(\g) - \dim(\h)$. 

\begin{exmp}\label{examples_cartan_geometries}
Any Klein geometry is trivially a Cartan geometry modelled on itself. It is enough to consider the principal $H$-bundle $G \arr G/H$; then the Maurer-Cartan form $\te \in \Om^1 (G,\g)$ satisfies the requirements. More general examples of Cartan geometries include Riemannian structures, affine structures, projective structures or conformal structures (see e.g.\ chapter 6-7-8 of \cite{Sha97} and chapter 4 of \cite{Cap09}).
%
\end{exmp}

One could notice that, in order for Definition \ref{def_Cartan_geometry} to make sense, the entire model geometry is not strictly necessary: one only needs $\g$ to contain $\h$ as a vector subspace and to be a $H$-representation (a priori unrelated with the adjoint representation on $\h$).
Indeed, the fact that the $H$-representation on $\g$ extends the adjoint representation on $\h$ follows automatically from the properties of the Cartan connection $\theta$. To see this, fix an arbitrary $p \in P$ and denote by $a_p: \h \arr T_p P$ the infinitesimal $H$-action on $P$, so that $v^\dagger_p = a_p (v)$ for every $v \in \h$. Then,
$$ h^{-1} \cdot v = h^{-1} \cdot \te_p (a_p (v)) = \te_{p \cdot h} (d_{p \cdot h} R_h a_p (v) ) =
\te_{p \cdot h} (a_{ph} (Ad_h (v) ) ) = Ad_h (v),$$
where in the third passage we used the well-known identity
$$d_{p \cdot h} R_h a_p (v) = a_{ph} (Ad_h (v) ).$$

On the other hand, the fact that $\g$ is itself a Lie algebra is not used at all in the definition of Cartan geometry: one of the reasons to ask it is to be able to define a notion of curvature via a Maurer-Cartan-like equation (see Remark \ref{curvature_Cartan_geometries}).

\begin{obs}[Tangent bundle of a Cartan geometry]\label{tangent_bundle_Cartan_geometry}
Let $(P,\te)$ be a Cartan geometry over $M$ modelled on $(\g,\h)$; then the tangent bundle of $M$ is isomorphic to the vector bundle associated to $P$ and the representation $\g/\h \in \Rep(H)$:
\begin{equation*}
TM \cong P[\g/\h] := (P \times \g/\h ) /H.
\end{equation*}
This is a well known result (see e.g.\ Theorem 3.15 of \cite[Chapter 5]{Sha97}), which will be relevant for us in the later sections. It follows from the fact that the tangent space $T_x M$ at any point $x = [p] \in M$ can be identified with the vector space $\g/\h$. Note that such an identification depends on the choice of a representative of $x$; for each $p \in P$ there is a canonical linear isomorphism $\phi_p: T_x M \arr \g/\h$, induced by the Cartan connection $\te_p: T_p P \arr \g$.
\end{obs}

Let us recall now a special class of Cartan geometries, to which one often restricts in order to prove more powerful results; in our story, it will be relevant when looking at $G$-structures. It is also worth mentioning that all the most common non-trivial Cartan geometries (e.g.\ those mentioned in Example \ref{examples_cartan_geometries}) belong to this class.

\begin{defn}\label{def_reductive}
 A Klein geometry $(G,H)$ is called {\bf reductive} if there exists a $H$-module $\l$ which is a complement of $\h = \Lie(H)$ in $\g = \Lie(G)$. Equivalently, one asks that the vector space splitting $\g = \h \oplus \g/\h$ is also a splitting of $H$-modules ($\l$ is isomorphic to $\g/\h$ as vector spaces). In particular, considering the adjoint representation on $\h$, $\g$ can be seen as a $H$-module as well.

Similarly, a Klein pair $(\g,\h)$ is called {\bf reductive} if there exists a $\h$-module $\l$ which is a complement of $\h$ in $\g$, i.e.\ $\g = \h \oplus \l$. In particular, considering the adjoint representation on $\h$, $\g$ can be seen as an $\h$-module as well.

A Cartan geometry is called {\bf reductive} if its model geometry is.
\end{defn}


In the reductive case, any Cartan connection $\theta \in \Om^1 (P,\g = \h \oplus \l)$ splits as the sum of two $H$-equivariant forms $\theta_\h \in \Om^1 (P,\h)$ and $\theta_\l \in \Om^1 (P,\l)$. It is easy to check that $\theta_\h (v^\dagger) = v$ for every $v \in \h$, so that $\theta_\h$ defines a connection on the principal $H$-bundle $P \to M$. On the other hand, $\theta_\l (v^\dagger)=0$, so $\ker(d\pi) \subseteq \ker(\theta_\l)$; the following section gives an interpretation of this phenomenon.

%

\subsection{Reductive Cartan geometries and connections on $G$-structures}

We review now the precise relation between Cartan geometries and another well known framework to study geometric structures: $G$-structures. The goal is to motivate the generalisation of Cartan geometries to Cartan bundles, introduced in the next section.

Let $G \subseteq GL(n,\RR)$ be a Lie subgroup; we recall that a {\bf $G$-structure} on an $n$-dimensional manifold $M$ is a reduction of the structure group of the principal $GL(n,\RR)$-bundle of frames $Fr(M) \arr M$ (see e.g.\ \cites{Cra16b, Kob95, Ste64} for more details).

\begin{defn}\label{def_tautological_form}
Let $\pi: P \arr M$ be a $G$-structure; its {\bf tautological form} $\te_{taut} \in \Om^1 (P,\RR^n)$ is defined as
$$ (\te_{taut})_p (v) := p^{-1} (d \pi (v) ),$$
where we interpret the frame $p \in P$ as a linear isomorphism $p: \RR^n \arr T_{\pi(p)} M$.
\end{defn}

The form $\te_{taut}$ has many properties: among the most important ones, it is $G$-equivariant, pointwise surjective, and satisfies $\ker(\te_{taut}) = \ker (d \pi)$.
 Moreover, it is the key ingredient to prove the following fundamental statement, which appeared first as Theorem 2 in \cite{Kob56}, and is discussed also in Appendix A.2 of \cite{Sha97} and Section 1.3 of \cite{Cap09}. Note that, from now on, we will use the letter $H$ for the Lie subgroup of $GL(n,\RR)$, to avoid misunderstandings with the notations introduced earlier in Definition \ref{def_Klein_geometry}.

\begin{thm}\label{correspondence_cartan_geometries_g_structures_with_connections}
Let $H \subseteq GL(n,\RR)$ be a Lie subgroup and $M$ an $n$-dimensional manifold. Then there is a bijective correspondence
$$ \left\{   \begin{array}{c}
    \text{(isomorphism classes of)}  \\
    \text{Cartan geometries over $M$} \\
    \text{modelled on $(H \ltimes \RR^n,H)$}
    \end{array} \right\} 
\tilde{\longleftrightarrow}
\left\{   \begin{array}{c}
    \text{(isomorphism classes of)}  \\
    \text{$H$-structures over $M$} \\
    \text{together with a principal connection}
    \end{array} \right\}. $$
\end{thm}

\begin{proof}
The correspondence is given as follows. Any Cartan geometry $(P,\te)$ as above is automatically reductive. Indeed, the Lie algebra of $G = H \ltimes \RR^n$ splits as $\g = \h \oplus \l$, with $\l = \RR^n$ and the standard $H$-action on $\l$ given by matrix multiplication. Accordingly, we can decompose the Cartan connection $\te \in \Om^1 (P, \g)$ into $\te_\h \in \Om^1 (P,\h)$ and $\te_{\l} \in \Om^1 (P, \l)$. As observed above, $\te_\h$ is a principal connection on $P$, while $\te_{\l}$ can be interpreted as the tautological form of a $H$-structure as follows.



Fixing a basis $(e_1, ..., e_n)$ of $\g/\h$, for any $p \in P$ we can consider the linear isomorphism $\phi_p: T_x M \arr \g/\h$ from Remark \ref{tangent_bundle_Cartan_geometry}, so that $(\phi_p^{-1} (e_1), ..., \phi_p^{-1} (e_n) )$ is a basis of $T_{[p]} M$. Denoting by $Q \subseteq Fr(M)$ the set of all frames of the form $(\phi_p^{-1} (e_1), ..., \phi_p^{-1} (e_n) )$, for any $p \in P$, one checks easily that $Q$ is a $H$-structure and $P \arr Q$ an isomorphism of principal bundles.
Then $P$ can be seen as a $H$-structure, and $\te_\l$ as its tautological form, identifying the vector space $\l$ with $\g/\h$.

Conversely, given a $H$-structure $P \subseteq Fr(M)$ and a connection $\ga \in \Om^1 (P,\h)$, we define a Cartan connection on $P$ as the sum $\te = \ga + \te_{taut} \in \Om^1 (P,\g)$, where $\te_{taut} \in \Om^1 (P,\RR^n)$ is the tautological form of $P$.
\end{proof}

\begin{obs}[Curvature and torsion]\label{curvature_Cartan_geometries}
In the setting of Proposition \ref{correspondence_cartan_geometries_g_structures_with_connections}, one gets further correspondences between other relevant objects. To any Cartan geometry $(P,\te)$ one associates its {\bf curvature} via the classical Maurer-Cartan formula
$$ \Om := d\te + \frac{1}{2} [\te, \te] \in \Om^2 (P, \g),$$
and its {\bf torsion} by taking the image of $\Om$ via the projection $\pr: \g \to \g/\h$
$$\tau := \pr (\Om) \in \Om^2 (P, \g/\h).$$
When $\theta$ is reductive, $\tau$ has a more explicit expression, and one can easily write down the precise relations between $\tau$ and the torsion of the connection $\te_\h$ from Proposition \ref{correspondence_cartan_geometries_g_structures_with_connections}; similarly for $\Omega$ and the curvature of $\theta_\h$ (see e.g.\ Theorem 3 of \cite{Kob56}). Moreover,
\begin{itemize}
\item if the homogeneous space $G/H$ is symmetric, i.e.\ $\l \cong \g/\h$ is a Lie algebra satisfying $[ \l, \l] \subseteq \h$, then the torsion of $(P,\te)$ coincides with the torsion of the connection $\te_\h$,
\item if, furthermore, $ [\l, \l ] = 0$ (i.e.\ $\l$ is an abelian Lie algebra), then also the curvature of $(P,\te)$ coincides with the curvature of the connection $\te_\h$.
\end{itemize}
A Cartan geometry is called {\bf flat} or {\bf torsion-free} if, respectively, its curvature or its torsion vanishes. For instance, if $H = GL(n,\RR)$, one has the affine space ${\mathds A}^n = H \ltimes \RR^n$; a Cartan geometry modelled on $({\mathds A}^n, H)$ is an affine geometry. Since $\RR^n = {\mathds A}^n/H$ is an abelian Lie algebra, such a Cartan geometry is flat and torsion-free precisely when the corresponding connection on $Fr(M)$ is flat and torsion-free, recovering the standard notion of {\it affine structure} on a manifold.
\end{obs}



The correspondence from Proposition \ref{correspondence_cartan_geometries_g_structures_with_connections} gives therefore a compact framework to investigate $H$-structures with connections, a topic extensively studied in the literature, via Cartan geometries. In order to treat even more examples, namely $H$-structures without the choice of a connection, from a similar point of view, we will introduce in the next section the more general framework of Cartan bundles.

\

We conclude this section by recalling that one has also the following slightly more conceptual notion, obtained by extracting the key properties from the tautological form of a $H$-structure:

\begin{defn}\label{definition_abstract_structure}
An {\bf abstract $H$-structure} $(P,\te)$ over $M$ consists of a principal $H$-bundle $\pi: P\arr M$ and a $H$-equivariant 1-form $\te \in \Om^1 (P, \RR^n)$ which is pointwise surjective and satisfies $\ker (\te) = \ker(d\pi)$. 
\end{defn}

Of course, ``concrete'' $H$-structures with their tautological forms $\theta_{taut}$ (Definition \ref{def_tautological_form}) are abstract $H$-structures. Conversely, given any abstract $H$-structure $(P,\te)$, one can produce an injective immersion of principal bundles $j: P \arr Fr(M)$ such that $ \te = j^* (\te_{taut})$ (this is well known; see \cite[Theorem 2.3.2]{Cat20} for an explicit proof). This means that there is basically no practical difference between the standard definition of $H$-structure and Definition \ref{definition_abstract_structure}; by adopting the latter, one avoids to keep track of unnecessary objects (the frame bundle).

\section{Cartan geometries and Lie groupoids}\label{section_Pfaffian}

In order to investigate Cartan geometries from the point of view of Lie groupoids, let us recall the following object.

\begin{defn}
Given a principal $G$-bundle $P \arr M$, its {\bf gauge groupoid} $\Gaug(P)$ is the quotient of the pair groupoid $P \times P \rightrightarrows P$ by the diagonal action of $G$:
\begin{equation*}
(P \times P)/G \rightrightarrows P/G \cong M. \qedhere
\end{equation*}
\end{defn}

The groupoid structure of $\Gaug(P)$ is given as follows: an arrow $[p,q] \in \Gaug(P)$ has source $[q]$ and target $[p]$, the multiplication is $[p,q] [q,r] = [p, r] $, the unit $1_{[p]} = [p,p]$, and the inverse $[p,q]^{-1} = [q,p]$. As a consequence, the isotropy groups of $\Gaug(P)$ are all isomorphic to $G$. Moreover, $\Gaug(P)$ is trivially transitive: for any two points $[q], [p] \in M$ there exists an arrow $[p,q] \in \Gaug(P)$ sending one to the other. Actually, it is well known that gauge groupoids exhaust all transitive Lie groupoids:
\begin{thm}\label{correspondence_transitive_groupoids_principal_bundles}
Let $\mathcal{G} \rightrightarrows M$ be a transitive Lie groupoid and fix a point $x\in M$; then the $s$-fibre $P = s^{-1} (x)$ defines a principal bundle $t: P \arr M$ with structure group the isotropy group $G= \G_x$, and the map
$$ \Gaug(P)\arr \G, \quad [g, h]\mapsto g\cdot h^{-1}$$
is an isomorphism of Lie groupoids. This induces a bijective correspondence:
$$ \left\{   \begin{array}{c}
    \text{(isomorphism classes of)}  \\
    \text{transitive Lie groupoids over $M$}
    \end{array} \right\} 
\tilde{\longleftrightarrow}
\left\{   \begin{array}{c}
    \text{(isomorphism classes of)}  \\
    \text{principal bundles over $M$}
    \end{array} \right\}. $$
\end{thm}

Our theorem below restricts the correspondence of Proposition \ref{correspondence_transitive_groupoids_principal_bundles} by considering on the right-hand side principal bundles with a structure of Cartan geometry on them.

\begin{thmm}\label{correspondence_pfaffian_groupoids_zero_symbol}
Let $(P,\te)$ be a Cartan geometry over $M$ modelled on $(\g,\h)$, with $\te \in \Om^1 (P, \g)$. Consider the gauge groupoid $\G = \Gaug(P)$ associated to $P$, and the representation $E = P[\g] \in \Rep(\G)$ induced from $\g \in \Rep(H)$. Then $\G$ is endowed with a form $\om \in \Om^1 (\G, t^*E)$ such that
\begin{enumerate}
\item $\om$ is multiplicative (Definition \ref{generalisation_multiplicative})
\item $\om$ is pointwise surjective
\item $\ker(ds) \cap \ker(\om) = \ker (dt) \cap \ker(\om) = 0$.
\end{enumerate}
 Conversely, any transitive Lie groupoid $\G$, endowed with a representation $E$ and a form $\om \in \Om^1 (\G, t^*E)$ satisfying the properties above, arises from a Cartan geometry (up to the choice of a model geometry, in the sense explained below).
\end{thmm}


Let us be more precise on the second part of the statement of Theorem \ref{correspondence_pfaffian_groupoids_zero_symbol}: in the proof below, starting with $(\G,E,\om)$, we will consider the pair $(\g:= E_x, \h:= \om_{1_x} (T_{1_x} H))$, where $H$ is the isotropy group of $\G$ at some $x \in M$. Technically speaking this is not a model geometry, but the only missing property is the fact that $\g$ is a Lie algebra. As explained in Section \ref{section_cartan_geometries}, this is not needed at all to define a Cartan geometry, and become relevant only e.g.\ if one looks at its curvature and torsion (which is not the case here). The last part of our statement reads therefore as follows: any $(\G, E, \om)$ satisfying the properties above arises from a Cartan geometry, up to the choice of a Lie bracket on a fibre of $E$.

We also warn the reader that, in the following proof, we will use several results from the Appendix.

\begin{proof}
We apply Proposition \ref{transportation_multiplicative_via_Morita} to the principal $H$-bundle $\pi: P \arr M$. Here we consider the Cartan connection $\te \in \Om^1 (P,\g)$ on $P$ and the Maurer-Cartan form $\omega_{MC} \in \Om^1 (H, \h)$ on $H$; since $\theta$ is $H$-equivariant and $\theta_p ( a_p (v) ) = v$ for every $v \in \h$, the $H$-action on $P$ is multiplicative (Corollary \ref{condition_multiplicativity_principal_bundle}). Accordingly, the following differential form $\om \in \Om^1 (\G, t^*E)$ is well defined and multiplicative:
$$ \om_{[p,q]} ([v,w]) := \te_p (v) - [p,q] \cdot \te_q (w).$$

Clearly, $\om$ is pointwise surjective since $\te$ is so. In order to check the third condition on $\om$, we use the second part of Proposition \ref{transportation_multiplicative_via_Morita} and the fact that $\te$ is pointwise injective:
$$ \pi^* (\ker(ds) \cap \ker(\om)_{\mid M} ) \cong \ker (\te) = 0.$$
The distribution $\ker(ds) \cap \ker(\om)$ is then zero at any unit $1_x \in \G$. Using Lemma \ref{multiplicative_form_is_equivariant}, we see that the right translation $R_g$, for any $g \in s^{-1}(x)$, induces an isomorphism $d_{1_x} R_g$ which sends $\ker(d_{1_x}s) \cap \ker(\om_{1_x})$ to $\ker(d_g s) \cap \ker (\om_g)$; hence $\ker(ds) \cap \ker(\om) = 0$ at any point of $\G$.

Last, observe that the inverse map $i$ of any Lie groupoid induces an isomorphism $di: \ker(ds) \arr \ker(dt)$. Since $\om$ is multiplicative, the formula $(i^* \om)_g = - g^{-1} \cdot \om_g$ holds for every $g \in \G$; it follows that $di$ sends $\ker(ds) \cap \ker(\om)$ isomorphically to $\ker(dt) \cap \ker(\om)$. We conclude that $\ker(dt) \cap \ker(\om) = 0$ as well.

\

Conversely, given a transitive Lie groupoid $\G$ and a 1-form $\om$ taking values in some $E \in \Rep(\G)$, fix any $x \in M$, and consider the isotropy group $H:= \G_x$ and the principal $H$-bundle $P:= s^{-1}(x)$ associated to $\G$ by Proposition \ref{correspondence_transitive_groupoids_principal_bundles}, with projection $\pi:= t_{\mid P}: P \arr M$. Under our hypotheses on $\om$, we are going to prove that $P$ carries the following structure of Cartan geometry modelled on $(\g:= E_x, \h:= \om_{1_x} (T_{1_x} H) )$:
$$\te \in \Om^1 (P,\g), \quad \quad \te_g (v) := g^{-1} \cdot \om_g (v).$$

First of all, we prove that $(\g,\h)$ is indeed a model geometry. As first step, we show that $\om$ induces a vector bundle isomorphism between the representation $E$ and the Lie algebroid $A := \Lie(\G)$. We already know that, for every $g \in \G$, the linear map
 $$\om_g: T_g \G \arr E_{t(g)}$$
 is surjective. Since $\om$ is also multiplicative, we can apply Lemma \ref{multiplicativity_implies_s_transversality}, and since $\ker(ds) \cap \ker(\om) = 0$, we obtain the decomposition
 $$ T_g \G = \ker(d_g s) \oplus \ker(\om_g). $$
 It follows that $\om_g$ remains surjective even when restricted to $\ker(d_g s)$, and that the kernel of ${\om_g}_{\mid \ker(d_g s) }$ is zero. In particular, for every $x \in M$, consider $g = 1_x$: then
$$ \om_{1_x}: A_x = \ker(d_{1_x}s) \arr E_x$$
 is a linear isomorphism. Because of this, we can not only identify $T_{1_x} H \subseteq A_x$ (the isotropy algebra of $A$ at $x$) with the subspace $\h:= \omega_{1_x} (T_{1_x} H) \subseteq \g$, but also transport the Lie algebra structure and the adjoint $H$-representation from $T_{1_x} H$ to $\h$.

As remarked after the statement of this theorem, we will ignore the Lie algebra structure of $\g$. Accordingly, to conclude the proof that $(\g,\h)$ is a model geometry, one needs only to show that the $H$-representation on $\g$ (induced by the $\G$-representation on $E$) is actually an extension of the adjoint $H$-representation on $\h$. But, as discussed in Section \ref{section_cartan_geometries}, this will be automatic once we show that $(P,\theta)$ is a Cartan geometry.

Let us prove now that $\theta$ is a Cartan connection. From Lemma \ref{multiplicative_form_is_equivariant} (based on the multiplicativity of $\om$), it follows that $\te$ is $H$-equivariant:
$$ ((R_h)^* \te)_g (v) = \te_{gh} (d R_h (v)) = (gh)^{-1} \cdot \om_{gh} (dR_h (v)) = $$
$$ = h^{-1} g^{-1} \cdot ((R_h)^*\om)_g (v) = h^{-1} \cdot g^{-1} \cdot \om_g (v) = h^{-1} \cdot \te_g (v).$$
Moreover, $\ker(\te)$ can be computed as
$$ \ker (\te) = TP \cap \ker(\om) = \ker(ds) \cap \ker(\om)_{\mid P} \cong \g(\om) = 0.$$
Here the last isomorphism is given by the right translations: using Lemma \ref{multiplicative_form_is_equivariant}, for every $p \in P$, with $\pi(p)=y$, the isomorphism $d_p R_{p^{-1}}$ sends $\ker(d_p s) \cap \ker(\om_p)$ to $\ker(d_{1_y} s) \cap \ker(\om_{1_y})$. This proves that $\te$ is pointwise injective.

Since $\om$ is pointwise surjective, for every non-zero element $\xi \in E_{s(g)}$, there exists some $v \in T_g \G$ such that $\om_g (v) = g \cdot \xi \in E_{t(g)}$. Actually, since $T_g \G = \ker(d_g s) \oplus \ker(\om_g)$ and $\om_g(v) \neq 0$, there exists some $v' \in \ker(d_g s) = T_g P$ such that $\om_g(v)=\om_g(v')$; therefore, $$\te_g (v') = g^{-1} \cdot \om_g (v') = g^{-1} \cdot \om_g (v) = \xi.$$
We conclude that $\te$ is pointwise surjective as well, hence it is a pointwise isomorphism.


Last, we prove the third property of Definition \ref{def_Cartan_geometry}, namely that $\theta(v^\dagger) = v$ for every $v \in \h$. Note that $v = \om_{1_x}(\alpha)$ for some $\alpha \in T_{1_x} H$, and therefore (paying attention to the difference between $\h$ and $T_{1_x} H$) $v^\dagger_g = a_g (\alpha) = d_{1_x} m(g,\cdot) (\alpha )$. Accordingly, for any $g \in P \subseteq \G$,
$$ \te_g (v^\dagger_g) = g^{-1} \cdot \om_g ( d_{1_x} m(g,\cdot) (\alpha ) ) = g^{-1} \cdot (L_g^*\om)_{1_x} (\alpha) = g^{-1} \cdot g \cdot \om_{1_x} (\alpha) = \om_{1_x} (\alpha) = v,$$
where we used the fact that the $H$-action on $P$ is the restriction of the multiplication $m$ of $\G$, and we applied Lemma \ref{multiplicative_form_is_equivariant} to $\omega$.
 \end{proof}


The pair $(\G,\om)$ that we have just described in Theorem \ref{correspondence_pfaffian_groupoids_zero_symbol} is an instance of the following object:

\begin{defn}\label{def_Pfaffian_groupoid}
A {\bf Pfaffian groupoid} $(\mathcal{G}, \om)$ over $M$ consists of a Lie groupoid $\mathcal{G} \rightrightarrows M$ together with a representation $E \arr M$ of $\mathcal{G}$ and a differential form $\om \in \Om^1 (\mathcal{G}, t^* E)$ such that
\begin{enumerate}
\item $\om$ is multiplicative (Definition \ref{generalisation_multiplicative})
\item $\om$ is of constant rank 
\item The subbundle
$$\mathfrak{g}(\om): = (\ker (ds) \cap \ker (\om))_{\mid M} \subseteq \Lie(\mathcal{G})$$
is a Lie subalgebroid of $\Lie(\G)$.
\end{enumerate}
We call $\g(\om)$ the {\bf symbol space} of $(\mathcal{G}, \om)$. Moreover, a Pfaffian groupoid $(\G,\om)$ is called
\begin{itemize}
\item {\bf full} if the form $\om$ is pointwise surjective.
\item {\bf Lie-Pfaffian}, or of Lie type, if it satisfies the additional condition
\begin{equation*}
\ker (dt) \cap \ker(\om) = \ker (ds) \cap \ker(\om). \qedhere 
\end{equation*}
\end{itemize}
\end{defn}
Notice that, with the same arguments involving right-translations used in Theorem \ref{correspondence_pfaffian_groupoids_zero_symbol}, one sees that $\g(\om) \subseteq \Lie(\G)$ is a subalgebroid if and only if $\ker(ds) \cap \ker(\om) \subseteq T\G$ an involutive distribution.

Of course, if $\g(\om) = 0$, the symbol space is automatically a Lie subalgebroid. Then Theorem \ref{correspondence_pfaffian_groupoids_zero_symbol} can be rephrased as follows:
$$ \left\{   \begin{array}{c}
    \text{(isomorphism classes of)}  \\
    \text{transitive full Lie-Pfaffian groupoids over $M$} \\
    \text{with zero symbol}
    \end{array} \right\} 
\tilde{\longleftrightarrow}
\left\{   \begin{array}{c}
    \text{(isomorphism classes of)}  \\
    \text{Cartan geometries over $M$} \\
    \text{(up to the model geometry)}
    \end{array} \right\}. $$


The notion of Pfaffian groupoid was first introduced in \cite{Sal13} in order to understand the structure behind the jet groupoids of Lie pseudogroups; see also \cite{Yud16} for a revisitation of Cartan's original works on pseudogroups in this framework. Equivalently, a Pfaffian groupoid can be interpreted as the Lie theoretic version of a {\it Pfaffian fibration}, a notion encoding the essential properties of PDEs on jet bundles together with their Cartan forms (see \cite{Cat19}).

\begin{obs}[relations with previous works on Cartan geometries]\label{obs_relation_other_works}
Our approach on Cartan geometries fits in some recent reformulations using the language of Lie groupoids and algebroids. Blaom introduced in \cite{Bla06} the notion of {\it Cartan algebroid}, i.e.\ a Lie algebroid $A$ together with a compatible connection. If $A$ is the Atiyah algebroid $TP/H$ associated to a principal bundle $P$, then it describes the infinitesimal counterpart of a Cartan connection $\te$ on $P$.
 Since $A$ is the Lie algebroid of the gauge groupod $\Gaug(P)$, the way to recover his result from our formalism is via the correspondence between multiplicative 1-forms on Lie groupoids and Spencer operators on Lie algebroids described in \cite{Cra12}.
 

Blaom introduced in \cite{Bla16} also the global counterpart of a transitive Cartan algebroid in terms of distributions on the gauge groupoid of $P$ which are compatible with the groupoid multiplication. In particular, our Theorem \ref{correspondence_pfaffian_groupoids_zero_symbol} resembles Blaom's \cite[theorem 1.1]{Bla16}; given our result, one can prove Blaom's by considering the distribution $\ker(\te)$.
We believe that our proof is more enlightening since it follows entirely from the general properties of multiplicative forms on Lie groupoids.

Last, we also mention the recent book \cite{Sau16} by Crampin and Saunders. They proposed a revised approach to Cartan geometries, introducing a notion of {\it infinitesimal Cartan connection} on a Lie algebroid, which generalises further Blaom's Cartan algebroids. However, little focus is given on the global counterpart of those objects.
\end{obs}

\begin{obs}[clarification on the roles of fullness and Lie-Pfaffian properties]
 Through this paper, every Pfaffian groupoid will actually satisfy the two additional properties described in Definition \ref{def_Pfaffian_groupoid}, namely being full and Lie-Pfaffian. It would therefore seem natural to include them in the general definition, as indeed it has been done in \cite{Sal13}, when Pfaffian groupoids appeared from the first time.

However, there are two main reasons to keep those conditions separate from the main definition: there are natural examples of Pfaffian groupoids which do not satisfy them, and such properties are not invariant under Morita equivalence (see \cite{Cat20} for more details). We are currently investigating in \cite{Acc20} the role of those conditions from the point of view of Cartan geometries (and Cartan bundles).
\end{obs}


\subsection{Cartan bundles}

Given the discussion in the previous section, we present now a generalisation of Cartan geometries which arises from transitive full Lie-Pfaffian groupoids with non-zero symbol.

\begin{defn}\label{definition_Cartan_bundle}
A {\bf Cartan bundle} $(P,\te)$ is a principal $H$-bundle $P \xrightarrow{\pi} M$, for $H$ a Lie group, together with a representation $V \in \Rep(H)$ and a differential form $\te \in \Om^1 (P, V)$ such that
\begin{itemize}
 \item $\te$ is pointwise surjective 
 \item $\ker(\te) \subseteq \ker(d\pi)$
 \item $\ker(\te)$ is an involutive distribution on $P$
 \item $\te$ is $H$-equivariant, i.e.\
 $ (R_h)^*\te = h^{-1} \cdot \te$
 for every $h \in H$
 \item for every $v \in \h$, the vector $\te_p (v^\dagger_p) \in V$ does not depend on $p \in P$. \qedhere
 
 
\end{itemize}
\end{defn}

It follows by the first two requirements that $\dim(V)$ is bounded by the dimensions of $M$ and of $P$. As anticipated, Definition \ref{definition_Cartan_bundle} has the following two extreme cases, when $\ker(\te)$ is the largest or the smallest possible distribution, and $V$ has the smallest or the largest possible dimension. 

\begin{exmp}\label{Cartan_geometry_is_cartan_bundle}
Let $P \to M$ be a principal $H$-bundle and $(\g,\h)$ be a model geometry. Then Cartan bundles on $P$ with $\ker(\theta)=0$ and $V=\g$ are the same thing as Cartan geometries modelled on $(\g,\h)$.

Indeed, starting with a Cartan bundle with $\ker(\theta)=0$, the form $\theta$ becomes a pointwise isomorphism. Accordingly, $\theta_p: T_p P \to \g$ restricts to an isomorphism $\phi_p: \ker (d_p \pi) \to \h$ by dimensional reasons. Then for every $v \in \h$ we have $v = \phi_p (w)$ for some $w \in \ker(d_p \pi)$, and in turn $w = a_p (\alpha)$ for some $\alpha \in T_e H$ since $\ker(d_p \pi) = \Ima(a_p)$. We conclude that $\theta_p (v^\dagger ) = \theta_p (a_p (\alpha)) = \phi_p (a_p (\alpha)) = v$, hence $(P,\theta)$ is a Cartan geometry.

Conversely, if we start with a Cartan geometry, the distribution $\ker(\te)$ is zero, so it is trivially involutive and inside the vertical bundle; moreover, $\te_p (v^\dagger_p) = v$, so it does not depend on $p$.
 \end{exmp}

 \begin{exmp}\label{H_structure_is_cartan_bundle}
Let $H \subseteq GL(n,\RR)$ be a Lie subgroup. Then Cartan bundles on $P$ with $V = \RR^n$ (with the natural matrix representation of $H$) and $\ker(\te)=\ker(d\pi)$ are the same thing as abstract $H$-structures (Definition \ref{definition_abstract_structure}).

Indeed, the involutivity of $\ker(\theta)$ comes for free, and $\te_p (a_p (v)) = 0$ for every $v \in \h$ (since $v^\dagger$ is a vertical vector field), so it does not depend on $p$.
\end{exmp}

\begin{obs}
Recall from Proposition \ref{correspondence_cartan_geometries_g_structures_with_connections} that a Cartan geometry modelled on $(H \ltimes \RR^n, H)$ can be viewed as a $H$-structure together with a connection; in the framework of Cartan bundles, we have therefore decoupled the $H$-structure from the connection. Note also that a principal $H$-bundle $P$ together with a connection form $\te \in \Om^1(P,\h)$ is of course not a Cartan bundle: it does not satisfy the condition $\ker(\te) \subseteq \ker(d\pi)$, and $\ker(\te)$ is involutive only if it the connection is flat.
\end{obs}

As promised, Cartan bundles extend the correspondence from Theorem \ref{correspondence_pfaffian_groupoids_zero_symbol} to the more general case of transitive full Lie-Pfaffian groupoids with any symbol:

\begin{thmm}\label{correspondence_cartan_bundles_pfaffian_groupoids}
For any manifold $M$ there is a 1-1 correspondence
$$ \left\{   \begin{array}{c}
    \text{(isomorphism classes of)}  \\
    \text{transitive full Lie-Pfaffian groupoids on $M$}
    \end{array} \right\} 
\tilde{\longleftrightarrow}
\left\{   \begin{array}{c}
    \text{(isomorphism classes of)}  \\
    \text{Cartan bundles on $M$} 
    \end{array} \right\}. $$
\end{thmm}

\begin{proof}
The proof goes like in Theorem \ref{correspondence_pfaffian_groupoids_zero_symbol}. Let $(\G,\om)$ be a transitive full Lie-Pfaffian groupoid, with $\om$ taking values in $E \in \Rep(\G)$, and fix any $x \in M$. Then $V: = E_x$ is a representation of the isotropy group $H:= \G_x$, and the principal $H$-bundle $P:= s^{-1}(x)$, with projection $\pi:= t_{\mid P}: P \arr M$, is a Cartan bundle with the differential form
$$\te \in \Om^1 (P,V), \quad \quad \te_g (v) := g^{-1} \cdot \om_g (v).$$
Let us show that $\te$ satisfies indeed the five properties of Definition \ref{definition_Cartan_bundle}. First, $\theta$ is pointwise surjective since $\om$ is. From Lemma \ref{multiplicative_form_is_equivariant} (based on the multiplicativity of $\om$), it follows that $\te$ is $H$-equivariant:
$$ ((R_h)^* \te)_g (v) = \te_{gh} (d R_h (v)) = (gh)^{-1} \cdot \om_{gh} (dR_h (v)) = $$
$$ = h^{-1} g^{-1} \cdot ((R_h)^*\om)_g (v) = h^{-1} \cdot g^{-1} \cdot \om_g (v) = h^{-1} \cdot \te_g (v).$$
Moreover, $\ker(\te)$ can be computed as
$$ \ker (\te) = TP \cap \ker(\om) = \ker(ds) \cap \ker(\om)_{\mid P}.$$
As in Theorem \ref{correspondence_pfaffian_groupoids_zero_symbol}, one uses the right translations of $\G$ to show that $\ker(\te) \cong \g(\om)$. In turn, the symbol space $\g(\om)$ is a Lie subalgebroid of $\Lie(\G)$, hence the distribution $\ker(\te)$ is involutive. Similarly, since $(\G, \om)$ is Lie-Pfaffian, $\g(\om) \subseteq \ker(dt)$, hence $\ker(\te)$ is contained in $\ker(dt) \cap TP = \ker(d\pi)$. 


Last, we check that, for every $v \in \h$, $\theta_g (v^\dagger_g)$ does not depend on $g \in P$:
$$ \theta_g (v^\dagger_g) = \theta_g (d_e m (g,\cdot) (v) ) = g^{-1} \cdot \omega_g (d_e L_g (v) ) = $$
$$ = g^{-1} \cdot (L_g^*\omega)_e (v) = g^{-1} \cdot g \cdot \omega_e (v) = \omega_e (v),$$
where in the second line we used Lemma \ref{multiplicative_form_is_equivariant}.

%
%
%

\

Conversely, consider a Cartan bundle $(P,\te)$; we are going to use Proposition \ref{transportation_multiplicative_via_Morita} in order to produce a multiplicative form $\om$ on the gauge groupoid $\G := (P \times P) /H$ with values in the representation $E := P[V]$. Consider the form $\omega_H \in \Om^1 (H, V)$ defined by
$$ (\omega_H)_g (\alpha) := \theta_p (a_p ( d_g R_{g^{-1}} (\alpha)) ) \quad \quad \text{for some } p \in P;$$
the expression does not depend on $p \in P$ since $(P,\theta)$ is a Cartan bundle, and one checks easily that $\omega_H$ is multiplicative. Since $\te$ is $H$-equivariant by hypothesis and $\theta_p (a_p(\alpha)) = (\omega_H)_e (\alpha)$ by construction, we conclude that the $H$-action on $P$ is multiplicative w.r.t.\ $\theta$ and $\omega_H$ (Proposition \ref{multiplicative_action_principal_bundles}). As anticipated, it follows from Proposition \ref{transportation_multiplicative_via_Morita} that the following differential form $\om \in \Om^1 (\G, t^*E)$:
$$ \om_{[p,q]} ([v,w]) := \te_p (v) - [p,q] \cdot \te_q (w)$$
is well defined and multiplicative.

The last part of Proposition \ref{transportation_multiplicative_via_Morita} guarantees also that
$$ \pi^* (\ker(ds) \cap \ker(\om)_{\mid M} ) \cong \ker(\te).$$
Since $\ker(\te)$ is involutive, $\g(\om) = \ker(ds) \cap \ker(\om)_{\mid M} \subseteq \Lie(\G)$ is involutive as well, hence is a subalgebroid.

It follows from the definition of $\om$ and the fact that $\ker(\theta) \subseteq \ker(d\pi)$ that
$$ \ker (ds) \cap \ker(\om) = [\ker(\te), \Ima(a)] = [\Ima(a), \ker(\te)] = \ker (dt) \cap \ker(\om),$$
hence $(\G,\om)$ is of Lie type. Last, $\om$ is pointwise surjective since $\te$ is so: indeed, for every $g=[p,q] \in \G$ and $v \in E_{t(g)} = E_{\pi(p)} = \te_p (T_p P)$, we have $\om_g ([\al,0]) = \te_p (\al) = v$ for some $\al \in T_p P$.
\end{proof}

In view of Example \ref{Cartan_geometry_is_cartan_bundle}, we see that Theorem \ref{correspondence_cartan_bundles_pfaffian_groupoids} fully recovers Theorem \ref{correspondence_pfaffian_groupoids_zero_symbol} in the case $\ker(\theta)=0$. On the other hand, in the setting of Example \ref{H_structure_is_cartan_bundle}, the maximal distribution $\ker(\theta) = \ker(d\pi)$ of a Cartan bundle $(P,\theta)$ corresponds to the maximal symbol space $\g(\omega)$ of a Lie-Pfaffian groupoid $(\G,\om)$, namely the isotropy Lie algebra bundle $\g(\omega) = \ker(\rho)$ (for $\rho: A \to TM$ the anchor map):
 $$ \left\{   \begin{array}{c}
    \text{(isomorphism classes of)}  \\
    \text{transitive full Lie-Pfaffian groupoids over $M$} \\
    \text{with symbol $\ker(\rho)$}
    \end{array} \right\} 
\tilde{\longleftrightarrow}
\left\{   \begin{array}{c}
    \text{(isomorphism classes of)}  \\
    \text{abstract $H$-structures over $M$}
    \end{array} \right\}. $$

 This correspondence should be compared with that of Proposition \ref{correspondence_cartan_geometries_g_structures_with_connections} and the discussion thereafter: as promised, we had introduced a Cartan-like structure in order to be able to consider (abstract) $H$-structures without the choice of a connection.
 
 It is also interesting to look at the differential form $\omega_H$ introduced in the previous proof in our two extreme cases of Cartan bundles:
\begin{itemize}
 \item when $(P,\theta)$ is a Cartan geometry, then $(\omega_H)_g (\alpha) = d_g R_{g^{-1}}(\alpha)$, i.e.\ $\omega_H$ is the Maurer-Cartan form of $H$;
 \item when $(P,\theta)$ is an abstract $H$-structure, then $\omega_H$ becomes the zero form since $\Ima(a) =\ker(d\pi) = \ker(\theta)$.
\end{itemize}

We conclude by studying some properties of the coefficients of a Cartan bundle.

\begin{thm}[Representation associated to a Cartan bundle]\label{representation_of_cartan_bundle}
 Let $(P,\te)$ be a Cartan bundle and $\G = \Gaug(P)$ the associated Pfaffian groupoid from Theorem \ref{correspondence_cartan_bundles_pfaffian_groupoids}.
 Then the fibre of the representation $E = P[V] \in \Rep(\G)$ splits as
 $$ E_x \cong T_x M \oplus T_{1_x} H / \g_{1_x}(\om),$$
where $\g(\om)$ is the symbol space of $(\G,\om)$ (Definition \ref{def_Pfaffian_groupoid}).
 Moreover, the linear $\G$-action on $E$ restricts to the following action on $TM$:
\begin{equation}\label{pfaffian_representation}
 g \cdot v = d_g t (\al), \quad \quad \forall g \in s^{-1}(x), v \in T_xM \tag{*}
\end{equation}
where $\al$ is any element of $\ker(\om_g)$ such that $v = d_g s(\al)$.
\end{thm}

We remark that the $\G$-representation \eqref{pfaffian_representation} is independent from the Cartan bundle structure of $P$. Indeed, any Lie-Pfaffian groupoid (not necessarily transitive) admits such a representation on the tangent space of its base.

\begin{proof}
For any $x = \pi(p) \in M$ it is immediate to check that
$$ E_x = V = \Ima(\te_p) \cong T_p P / \ker(\te_p) \cong T_p P / \ker(d_p\pi) \oplus \ker(d_p\pi) / \ker(\te_p) \cong $$
$$ \cong T_x M \oplus T_{1_x} \G_x / \g_{1_x}(\om) \cong T_x M \oplus T_{1_x} H / \g_{1_x}(\om).$$

For the second part, consider the standard representation of $\G = (P \times P)/H$ on $E = (P \times V)/H$
\begin{equation}\label{standard_representation}
[p,q] \cdot [q,z] = [p,z], \tag{**}
\end{equation}
and the projection of $E$ on $TM$
$$ \Phi: E \arr TM, \quad [q,z] \mapsto d_q \pi (w),$$
where $z = \te_q (w)$ for some $w \in T_q P$. We are going to prove that
$$ g \cdot \Phi (v) = \Phi (g \cdot v ),$$
for every $g = [p,q] \in \Gaug(P)$ and $v = [q,z] \in E$; here on the left-hand side of the equality we applied the representation \eqref{pfaffian_representation} on $TM$, on the right-hand one the representation \eqref{standard_representation} on $E$. Accordingly, the two sides of the equation become
$$ g \cdot \Phi (v) = g \cdot d_q \pi (w),$$
$$ \Phi (g \cdot v ) = \Phi ( [p,\al] ) = d_p \pi (w'),$$
for some $w \in T_q P$ and $w' \in T_p P$ such that $z = \te_q (w) = \te_p (w')$.

Consider now the vector $\be = [w',w] \in T_g \G$; by construction, $\be \in \ker(\om_g)$ and $d_g s (\be) = d_q \pi(w)$. This implies that
$$ g \cdot d_q \pi (w) = d_g t (\be) = d_p \pi (w'),$$
which concludes our proof.
\end{proof}

\begin{exmp}
It is interesting to describe the splitting of the representation $E$ from Proposition \ref{representation_of_cartan_bundle} in the two particular cases we have examined. For a Cartan geometry $(P,\te)$ (Example \ref{Cartan_geometry_is_cartan_bundle}), the symbol space $\g(\om)$ of the associated Pfaffian groupoid $(\G,\om)$ is zero, so that the fibre of its representation is
$$ E_x = \h \oplus T_x M.$$
In the case when $(P,\theta)$ is reductive (Definition \ref{def_reductive}), this can also be seen directly: since $\g = \h \oplus \l$ splits as the sum of $H$-modules, the entire representation $E = P[\g]$ splits as the sum of $H$-representations as well
$$ E = P[\h] \oplus P[\l] \cong P[\h] \oplus TM,$$
where we identified $\l$ with $\g/\h$ and used Remark \ref{tangent_bundle_Cartan_geometry}.

On the other hand, for a $H$-structure $(P,\te)$ (Example \ref{H_structure_is_cartan_bundle}), the symbol space $\g(\om)$ of the associated Pfaffian groupoid $(\G,\om)$ is the isotropy algebra bundle $\ker(\rho)$, which has as standard fibre the Lie algebra of $H$, so that the term $T_{1_x} H / \g(\om)_{1_x}$ disappears and
\begin{equation*}
E = TM. \qedhere
\end{equation*}

Inspired by these extreme examples, one could also define the class of \textit{reductive} Pfaffian groupoids, described by suitable splittings of $E$ into $TM$ and another (possibly zero) representation. This will be object of a future work, which will lead to a theory of \textit{reductive} Cartan bundles, which aims at proving results analoguous to Remark \ref{tangent_bundle_Cartan_geometry} and at describing ``intermediate'' examples of Cartan bundles, which are not necessarily Cartan geometries or $G$-structures.
\end{exmp}


We conclude this paper by mentioning that we plan to further investigate in \cite{Acc20} the infinitesimal object associated to a Cartan bundle, building up on the works mentioned in Remark \ref{obs_relation_other_works}, e.g.\ on Cartan algebroids \cite{Bla06} and Spencer operators \cite{Cra12}. Indeed, it is our intention to develop a Lie theory for Cartan bundles, which will allows us to ``differentiate'' the global objects and ``integrate'' their infinitesimal counterpart, exactly like one does for standard Lie groupoids and Lie algebroids. This approach will be particularly advantageous to tackle (difficult) global problems by solving the corresponding (easier) infinitesimal ones.


\appendix

\section{Appendix}

In this appendix, necessary to understand the proofs of Section \ref{section_Pfaffian}, we collected some basic definitions and properties of multiplicative forms on Lie groupoids, as well as of Lie groupoid actions compatible with a multiplicative form. Some results are not standard, and constitute a particular case of general statements proved in the author's PhD thesis \cite{Cat20}.

\subsection{Multiplicative forms}


\begin{defn}\label{multiplicative_form}
Let $\mathcal{G}$ be a Lie groupoid; a differential form $\om \in \Om^k (\mathcal{G})$ is called {\bf multiplicative} if
$$m^* \om = \pr_1^* \om + \pr_2^* \om,$$
where $m: \mathcal{G} \tensor[_s]{\times}{_t} \mathcal{G} \subseteq \mathcal{G} \times \mathcal{G} \arr \mathcal{G}$ is the multiplication of $\mathcal{G}$ and $\pr_i: \mathcal{G} \tensor[_s]{\times}{_t} \mathcal{G} \arr \mathcal{G}$ are the projections on the $i^\text{th}$-component.
\end{defn}

Multiplicative forms arise naturally in many geometric contexts, e.g.\ to study symplectic or contact structures on Lie groupoids.
In this paper we consider forms with coefficients; to make sense of the multiplicativity condition, the coefficient must be the pullback bundle $t^*E$ of a representation $E$ of $\G$.


\begin{defn}\label{generalisation_multiplicative}
Let $\mathcal{G}$ be a Lie groupoid and $E$ a representation of $\G$; a differential form $\om \in \Om^k (\mathcal{G}, t^* E)$ is called {\bf multiplicative} if
\begin{equation*}
(m^* \om )_{(g,h)} = (\pr_1^* \om)_{(g,h)} + g \cdot (\pr_2^* \om)_{(g,h)} \quad \forall (g, h) \in \mathcal{G} \tensor[_s]{\times}{_t} \mathcal{G}.
\end{equation*}
To keep the notation simple, we will often write
\begin{equation*}
m^*\om = \pr_1^*\om + g \cdot \pr_2^* \om. \qedhere
\end{equation*}
\end{defn}


Here are two simple but fundamental properties of multiplicative 1-forms.

\begin{lem}\label{multiplicative_form_is_equivariant}
Let $\mathcal{G}$ be a Lie groupoid, $E$ a representation and $\om \in \Om^1(\mathcal{G}, t^*E)$ a multiplicative form. Then, for every arrow $g \in \mathcal{G}$ from $x$ to $y$:
\begin{itemize}
 \item 
 $ (L_g)^* (\om_{| t^{-1}(y)}) = g \cdot \om_{| t^{-1}(x)}, $
 \item 
 $ (R_g)^* (\om_{| s^{-1}(x)}) = \om_{| s^{-1}(y)}.$
\end{itemize}
\end{lem}

\begin{proof}
For any $(g,h) \in \mathcal{G} \tensor[_s]{\times}{_t} \mathcal{G}$ and $Y \in T_h (t^{-1} (s(g)))$, we have
$$ d_h L_g (Y) = d_h m(g, \cdot) (Y) = d_{(g,h)} m_{| T( \{g\} \times t^{-1} (s(g)) )} (0,Y),$$
where the last equality comes from a straightforward computation using tangent curves. Therefore, using the multiplicativity of $\om$, we obtain
$$ ((L_g)^* \om)_h (Y) = \om_{g \cdot h} (d_h L_g (Y) ) = $$
$$ = \om_{m (g,h)} (d_{(g,h)} m_{| T( \{g\} \times t^{-1} (s(g)) )} (0,Y) ) = \cancel{\om_g (0)} + g \cdot \om_h (Y). $$
With the same arguments, for any $(h, g) \in \mathcal{G} \tensor[_s]{\times}{_t} \mathcal{G}$ and $X \in T_g (s^{-1} (t(g)))$ we have
$$ d_h R_g (X) = d_h m(\cdot, g) (X) = d_{(h,g)} m_{| T( s^{-1} (t(g)) \times \{ g \} ) } (X,0),$$
and we conclude that
\begin{equation*}
 ((R_g)^* \om)_h (X) = \om_h (X) + \cancel{h \cdot \om_g (0)}. \qedhere 
\end{equation*}
\end{proof}

\begin{lem}\label{multiplicativity_implies_s_transversality}
Let $\mathcal{G}$ be a Lie groupoid, $E$ a representation and $\om \in \Om^1(\mathcal{G}, t^*E)$ a multiplicative form. If $\om$ is of constant rank, then it is $s$-transversal, i.e.\
$$ \ker(\om_g) + \ker(d_g s) = T_g \mathcal{G} \quad \forall g \in \mathcal{G},$$
as well as $t$-transversal, i.e.\
$$ \ker(\om_g) + \ker(d_g t) = T_g \mathcal{G} \quad \forall g \in \mathcal{G}.$$
\end{lem}

\begin{proof}
We prove first the $s$-transversality at the units of $\G$. For any $v \in T_{1_x} \mathcal{G}$ we consider $v' = d_x u (d_{1_x} s (v) )$, where $u: M \arr \G$ is the unit map $x \mapsto 1_x$. Using the multiplicativity of $\om$, it is immediate to check that $u^* \om = 0$, hence $v' \in \ker(\om_{1_x})$. We can therefore write
$$ v = v' + (v - v'),$$
and this proves the $s$-transversality at the units, since $v - v' \in \ker (d_{1_x} s)$:
$$d_{1_x} s (v - v') = d_{1_x} s (v) - d_{1_x} (s \circ u \circ s ) (v) = d_{1_x} s (v) - d_{1_x} (id_M \circ s) (v) = 0.$$

For the general statement, let us denote by $\cap_g$ the intersection
$$ \cap_g := \ker(\om_g) \cap \ker (d_g s), \quad g \in \mathcal{G}.$$
From what was said above, for $g=1_x$, the dimension of $\cap_{1_x}$ is constant:
$$\dim (\cap_{1_x}) = \dim(\ker(\om_{1_x})) + \dim (\ker (d_{1_x} s) ) - \dim (T_{1_x} \mathcal{G}).$$ 
On the other hand, using Lemma \ref{multiplicative_form_is_equivariant}, one sees that
$$ d_g {R_h}_{| \cap_g}: \cap_g \arr \cap_h \quad \forall (g,h) \in \mathcal{G} \tensor[_s]{\times}{_t} \mathcal{G}$$
is an isomorphism. Then the dimension of $\cap_g$ is constant:
$$ \dim (\cap_{g_1}) = \dim (\cap_{1_{s(g_1)}}) = \dim (\cap_{1_{s(g_2)}}) = \dim (\cap_{g_2}) \quad \forall g_1, g_2 \in \mathcal{G}.$$
Since $\om$ has constant rank, it follows that
$$\dim (\ker(\om_g)) + \dim (\ker (d_g s)) - \dim (\cap_g) = \const \quad \forall g \in \mathcal{G}.$$
In particular, $\const$ is the value taken at $g = 1_x$, which is equal to $\dim (\mathcal{G})$; then the $s$-transversality condition holds at any point.

Last, we show that $s$- and $t$-transversality are equivalent. Indeed, assuming that $\om$ is $s$-transversal, for every $g \in \mathcal{G}$, $d_g s$ is surjective when restricted to $\ker(\om_g)$. However, since $t = s \circ i$,
$$d_{g^{-1}} t_{| \ker(\om_{g^{-1}}) } = d_g s_{| \ker(\om_g)} \circ d_{g^{-1}} i_{| \ker(\om_{g^{-1}})},$$
where we used the multiplicativity of $\om$ to see that $d_g i$ sends $\ker(\om_{g^{-1}})$ into $\ker(\om_g)$. Being a composition of surjective maps, $d_{g^{-1}} t_{| \ker(\om_{g^{-1}})}$ is surjective as well for every $g \in \mathcal{G}$, hence $\om$ is $t$-transversal.
\end{proof}

\subsection{Multiplicative groupoid actions}

Let $\mathcal{G} \rightrightarrows M$ be a Lie groupoid acting (on the left) on the manifold $P$ along the map $\mu: P \arr M$; denote by $m_P$ the action map, defined on the fibred product
$$\mathcal{G} \tensor[_s]{\times}{_\mu} P:= \{ (g,p) \in \mathcal{G} \times P \mid s(g) = \mu(p) \}.$$
Moreover, let $E$ be a representation of $\mathcal{G}$, $\om \in \Om^k (\mathcal{G}, t^* E)$ a multiplicative form and $\te \in \Om^k (P, \mu^* E)$ a differential form; we represent this setting in the following diagram:

\begin{center}
  \tikzstyle{line} = [draw, -latex']
\pgfmathsetmacro{\shift}{0.4ex}
\begin{tikzpicture}[node distance = 1cm, auto]
    \node (G1) {$ (\mathcal{G}, \om) $};
    \node [right of=G1, scale=2] (1) {$\curvearrowright$};
    \node [right of=1] (P) {$(P, \te)$};
    \node [below of=G1] (M1) {$M$};
    
    \path [line] (P) -- (M1) node [midway,below] {$\mu$};

\draw[->,transform canvas={xshift=-\shift}](G1) to node[midway,left] {$s$}(M1);
\draw[->,transform canvas={xshift=\shift}] (G1) to node[midway,right]  {$t$} (M1);

\end{tikzpicture}
\end{center}

\begin{defn}\label{def_multiplicative_action}
A $\mathcal{G}$-action on $P$ as above is called {\bf multiplicative} (with respect to $\om$ and $\te$) if
$$ (m_P^* \te)_{(g, p)} = (\pr_1^* \om)_{(g,p)} + g \cdot (\pr_2^* \te)_{(g,p)} \quad \forall (g, p) \in \mathcal{G} \tensor[_s]{\times}{_\mu} P.$$
As for multiplicative forms, we will often denote this as
$$ m_P^* \te = \pr_1^* \om + g \cdot \pr_2^* \te.$$
Multiplicative right actions are defined analogously, with the condition
\begin{equation*}
 g \cdot (m_P^* \te)_{(p, g)} = (\pr_1^* \te)_{(p, g)} + (\pr_2^* \om)_{(p, g)} \quad \forall (p, g) \in P \tensor[_\mu]{\times}{_t}\mathcal{G}. \qedhere 
\end{equation*}
\end{defn}


In the particular case when $\G$ is a Lie group, Definition \ref{def_multiplicative_action} simplifies considerably.

\begin{thm}\label{multiplicative_action_principal_bundles}
Let $G$ be a Lie group acting on a manifold $P$, $V$ a representation of $G$, and consider a multiplicative form $\omega \in \Omega^1 (G,V)$ and a $G$-equivariant form $\theta \in \Omega^1 (P,V)$. Then the following three statements are equivalent.
\begin{enumerate}
 \item $(m_P^*\theta)_{(g,p)} (\alpha, v) = (\pr_1^* \omega)_{(g,p)}(\alpha, v) + g \cdot (\pr_2^*\theta)_{(g,p)} (\alpha, v)$ for every $\alpha \in T_g G$, $v \in T_p P$.
 \item $(m_P^*\theta)_{(e,p)} (\alpha, v) = (\pr_1^* \omega)_{(e,p)}(\alpha, v) + (\pr_2^*\theta)_{(e,p)} (\alpha, v)$ for every $\alpha \in \g$, $v \in T_p P$.
 \item $\theta_p (a_p (\alpha)) = \omega_e (\alpha)$ for every $\alpha \in \g$, $p \in P$, with $a_p: \g \to T_p P$ the infinitesimal $G$-action.
\end{enumerate}
\end{thm}

\begin{proof}
It is immediate to check that $1 \then 2 \then 3$ (by plugging in, respectively, $g=e$ and $v=0$). Conversely, to prove $3 \then 2$, it is enough to compute , for every $\alpha \in \g$ and $v \in T_p P$,
$$ (m_P^*\theta)_{(e,p)} (\alpha, v) - (\pr_2^*\theta)_{(e,p)} (\alpha, v) = \theta_p (d_{(e,p)} m_P (\alpha, v) - d_p L_e (v) ) = \theta_p (d_{(e,p)} m_P (\alpha, v) - d_{(e,p)} m_P (0, v) ) = $$
$$ = \theta_p (d_{(e,p)} m_P (\alpha, 0) ) = \theta_p (a_p (\alpha) ) = \omega_e (\alpha) = (\pr_1^* \omega)_{(e,p)}(\alpha).$$
Last, in order to show that $2 \then 1$, consider $\alpha \in T_g G$ and its left-translated $\beta:= d_g L_{g^{-1}} (\alpha) \in \g$, so that $\alpha = d_e L_g (\beta)$. It follows that
$$ \theta_{g \cdot p} (d_{(g,p)} m_P (\alpha, 0) ) = \theta_{g \cdot p} (d_{(g,p)} m_P ( d_e L_g (\beta), 0)  ) = \theta_{g \cdot p} (d_p L_g  (d_{(e,p)} m_P ( \beta, 0 ) )  ) = $$
$$ = (L_g^* \theta)_p (d_{(e,p)} m_P ( \beta, 0 ) ) = g\cdot \theta_p (d_{(e,p)} m_P ( \beta, 0 ) ) = g \cdot \theta_p (a_p (\beta) ) = $$
$$ = g \cdot \omega_e (\beta) = \omega_g ( d_e L_g (\beta) ) = (\pr_1^* \omega)_{(g,p)}(\alpha, v).$$
where in the second line we used the $G$-equivariance of $\theta$ and in the third property 2. We conclude, using again the $G$-equivariance of $\theta$, that for every $v \in T_p P$
$$ (m_P^*\theta)_{(g,p)} (\alpha, v) - (\pr_1^* \omega)_{(g,p)}(\alpha, v) = \theta_{g \cdot p} (d m_P (\alpha, v) ) - \theta_{g \cdot p} (d_{(g,p)} m_P (\alpha, 0) ) = $$
\begin{equation*}
= \theta_{g \cdot p} (d m_P (0, v) ) = \theta_{g \cdot p} (d_p L_g (v) ) = g \cdot \theta_p (v) = g \cdot (\pr_1^*\theta)_{(g,p)} (\alpha, v). \qedhere 
\end{equation*}
\end{proof}

\begin{cor}\label{condition_multiplicativity_principal_bundle}
 In the setting of Proposition \ref{multiplicative_action_principal_bundles},
\begin{itemize}
 \item if $\omega$ is the zero form, the action is multiplicative if and only if $\theta (\alpha^\dagger) = 0$ for every $\alpha \in \g$;
 \item if $V = \g$ and $\omega$ is the Maurer-Cartan form, the action is multiplicative if and only if $\theta(\alpha^\dagger) = \alpha$ for every $\alpha \in \g$, i.e.\ $\theta$ is a principal connection on $P$.
\end{itemize}
\end{cor}

The interested reader could find a similar infinitesimal condition for multiplicative actions of Lie groupoids (not just of Lie groups) in \cite[Proposition 5.3.4]{Cat20}.


\subsection{Principal bundles and multiplicative actions}

We are going to show how to use principal bundles to ``transport'' multiplicative forms via multiplicative actions. 

%
%
%

\begin{thm}\label{transportation_multiplicative_via_Morita}
Let $H$ be a Lie group and $\om_H \in \Om^1 (H, V)$ a multiplicative form with coefficients in a representation $V$ of $H$. Let also $P$ be a (left) principal $H$-bundle over $M$, and $\te \in \Om^1 (P, V)$ a differential form such that the $H$-action is multiplicative. 

Then $\Gaug(P)$ carries a unique multiplicative form $\om \in \Om^1 ( \Gaug(P), t^* P[V])$
such that
$$ \tau^* \om = \tilde{s}^* \te - \tilde{t}^* \te,$$
where $\tau$ is the projection $P \times P \arr \Gaug(P)$ and $\tilde{s}, \tilde{t}$ are the source and the target of the pair groupoid $P \times P \rightrightarrows P$. Moreover,
\begin{itemize}
 \item the action of $(\Gaug(P), \om)$ on $(P,\te)$ is multiplicative;
 \item the $\pi$-pullback of the vector bundle $\g(\om) := \ker(\om)\cap \ker(ds)_{\mid M}$ is isomorphic to $\ker(\theta)$, i.e.\
 $$ (\pi^*\mathfrak{g}(\om))_{p} \cong \ker(\te_p) \quad \forall p \in P.$$
\end{itemize}

\end{thm}

\begin{proof}
Let us represent on the following diagram the spaces and the maps we are going to use.

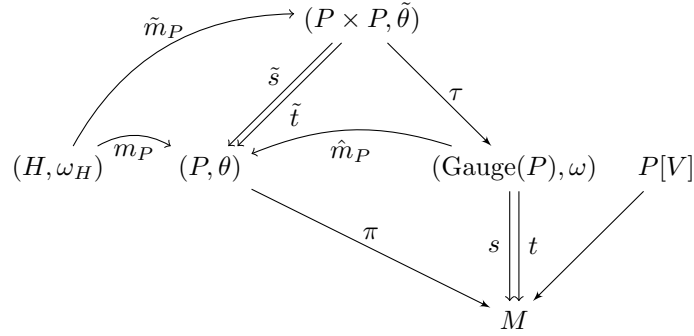
\begin{center}

  \tikzstyle{line} = [draw, -latex']
\begin{tikzpicture}[node distance = 2cm, auto]

\pgfmathsetmacro{\shift}{0.4ex}

    \node (0) {$ \quad $};
    \node [below of=0] (G) {$ (H, \om_H) $};
    \node [right of=G] (P) {$ (P, \te) $};
    \node [right of=P] (1) {$ \quad $};
    \node [above of=1] (PxP) {$ (P \times P, \tilde{\te}) $};
    \node [right of=1] (PxP/G) {$ (\Gaug(P), \om) $};
    \node [below of=PxP/G] (P/G) {$ M $};

 \node [right of=PxP/G] (repr) {$ P[V] $};

    \path[->] (G) edge [bend left=30] node[midway,below] {$ m_P $} (P);
    \path[->] (G) edge [bend left=30] node[midway,above] {$ \tilde{m}_P $} (PxP);

    \path[->] (PxP/G) edge [bend right=20] node[midway,below] {$ \hat{m}_P $} (P);

    \path [line] (PxP) -- (PxP/G) node [midway,right] {$ \tau $};

    
\draw[->,transform canvas={xshift=-\shift}](PxP) to node { $\tilde{t} $}(P);
\draw[->,transform canvas={xshift=\shift}] (PxP) to node[swap] {$ \tilde{s} $} (P);

\draw[->,transform canvas={xshift=-\shift}](PxP/G) to node[midway,left] {$ s $}(P/G);
\draw[->,transform canvas={xshift=\shift}] (PxP/G) to node[midway,right]  {$ t $} (P/G);

    \path [line] (P) -- (P/G) node [midway,above] {$ \pi $};
    \path [line] (repr) -- (P/G) node [midway,above] {$ $};

\end{tikzpicture}

\end{center}

The proof is carried out in five steps:
\begin{enumerate}
\item The form $\tilde{\te} := \tilde{s}^* \te - \tilde{t}^* \te \in \Om^1 (P \times P, V)$ is basic.
\item There is a unique form $\om \in \Om^1 ( \Gaug(P), t^* P[V])$ such that $\tau^* \om = \tilde{\te}$.
\item The form $\om$ is multiplicative.
\item The action of $\Gaug(P)$ on $P$ is multiplicative w.r.t.\ $\om$ and $\te$.
\item $ (\pi^*\mathfrak{g}(\om))_{p} \cong \ker(\te_p)$ for every $p \in P$.
\end{enumerate}

\underline{First part}: we denote by $\pr$ the projections from $H \times P$ on the first and second component, and by $\tilde{\pr}$ the projections from $H \times (P \times P)$ to either one of the three components or two of them. Using the multiplicativity of $m_P$ we find
$$ (\tilde{m}_P)^* \tilde{\te} = (\tilde{m}_P)^* (\tilde{s}^* \te ) - (\tilde{m}_P)^* (\tilde{t}^* \te ) = (\tilde{s} \circ (\tilde{m}_P))^* \te - (\tilde{t} \circ (\tilde{m}_P))^* \te = $$
$$ = (m_P \circ \tilde{\pr}_{13} )^* \te - (m_P \circ \tilde{\pr}_{12} )^* \te = \tilde{\pr}_{13}^* (m_P^* \te) - \tilde{\pr}_{12}^* (m_P^* \te) = $$
$$ = \tilde{\pr}_{13}^* (\pr_1^* \om ) + \tilde{\pr}_{13}^* (g \cdot \pr_2^* \te ) - \tilde{\pr}_{12}^* (\pr_1^* \om ) - \tilde{\pr}_{12}^* (g \cdot \pr_2^* \te ) = $$
$$ = \cancel{\tilde{\pr}_1^* \om} + g \cdot \tilde{\pr}_3^* \te - \cancel{\tilde{\pr}_1^* \om} - g \cdot \tilde{\pr}_2^* \te = $$
$$ = g \cdot ( (\tilde{s} \circ \tilde{\pr}_{23})^* \te - (\tilde{t} \circ \tilde{\pr}_{23})^* \te ) = g \cdot \tilde{\pr}_2^* (\tilde{s}^* \te - \tilde{t}^* \te) = g \cdot \tilde{\pr}_2^* \tilde{\te}. $$
We conclude that $\tilde{\te}$ is $H$-equivariant and horizontal, i.e.\ basic:
$$ (L_g^*\theta)_p (v) = \theta_{g \cdot p} (d_p \tilde{m}_P (g, \cdot) (v) ) = \theta_{g \cdot p} (d_{(g, p)} \tilde{m}_P (0, v) ) = (\tilde{m}_P^*\theta)_{(g,p)} (0,v) = g \cdot \theta_p (v), $$
$$ \theta_p (a_p(\alpha)) = \theta_p (d_e \tilde{m}_P (\cdot, p) (\alpha) ) = \theta_p (d_{(e,p)} \tilde{m}_P (\alpha,0) ) = (\tilde{m}_P^*\theta)_{(e,p)} (\alpha,0) = e \cdot \theta_p(0) = 0.$$

\

\underline{Second part}: since $\tau: P \times P \to \Gaug(P)$ is a principal bundle, we know from the general theory of basic forms that
$$ \Om^k (\Gaug(P), t^*P[V]) \arr \Om^k_{bas} (P \times P, V), \quad \alpha \mapsto \tau^* \alpha$$
is an isomorphism. Since $\tilde{\te} \in \Om^1_{bas} (P \times P, V )$, then there exists a unique form $\om \in \Om^1 ( \Gaug(P), t^*P[V] )$ such that $\tau^* \om = \tilde{\te}$.

\

\underline{Third part}: it is immediate to check that $\tilde{\te} = \tilde{s}^* \te - \tilde{t}^* \te$ is multiplicative. Moreover, denote by $\bar{m}$, $\bar{\pr_1}$ and $\bar{\pr_2}$ the maps
$$(P \times P) \tensor[_{\tilde{t}}]{\times}{_{\tilde{s}}} (P \times P) \arr P \times P$$
corresponding to the multiplication of the groupoid $P \times P$ and to the projections of $(P \times P) \tensor[_{\tilde{t}}]{\times}{_{\tilde{s}}} (P \times P)$ on the first and second component, and by $[ \bar{m} ]$ and $[ \bar{\pr}_i ]$ the projections of those maps to the quotient $(P \times P)/H$. 
With the usual arguments we get
$$ (\tau \times \tau)^* ( [ \bar{m} ]^* \om ) = ( [ \bar{m}] \circ (\tau \times \tau) )^* \om = (\tau \circ \bar{m})^* \om = \bar{m}^* (\tau^* \om) = \bar{m}^* \tilde{\te} = $$
$$ = \bar{\pr}_1^* \tilde{\te} + \bar{\pr}_2^* \tilde{\te} = \bar{\pr}_1^* (\tau^* \om) + \bar{\pr}_2^* (\tau^* \om) = (\tau \circ \bar{\pr}_1 )^* \om + (\tau \circ \bar{\pr}_2)^* \om = $$
$$ = ([\bar{\pr}_1] \circ (\tau \times \tau) )^* \om + ([\bar{\pr}_2] \circ (\tau \times \tau) )^* \om = (\tau \times \tau)^* ( [\bar{\pr}_1] ^* \om + [\bar{\pr}_2]^* \om ). $$
By the injectivity of the pullback we get the multiplicativity of $\om$.

\

\underline{Fourth part}: we see first that the action of the pair groupoid $P \times P$ on $P$,
$$ \hat{m}_P: P 
\tensor[_{id}]{\times}{_{\tilde{t}}} (P \times P) \arr P, \quad (p, (p,q) ) \mapsto q,$$
is multiplicative with respect to $\te$ and $\tilde{\te}$:
$$ (\hat{\pr}_1)^*\te + (\hat{\pr}_2)^* \tilde{\te} = \cancel{(\hat{\pr}_1)^*\te} + (\tilde{s} \circ \hat{\pr}_2)^* \te - \cancel{(\tilde{t} \circ \hat{\pr}_2)^* \te} = (\hat{m}_P)^*\te.$$

Then, when looking at the action $[\hat{m}_P]$ of the quotient $(P \times P)/H$ on $P$, the multiplicativity condition is preserved: 
$$ (id_P, \tau)^* ([\hat{m}_P]^* \te) = ( [\hat{m}_P] \circ (id_P,\tau) )^* \te = \hat{m}_P^* \te = \hat{\pr}_1^* \te - \hat{\pr}_2^*\tilde{\te} = $$
$$ = ( [\hat{\pr}_1] \circ (id_P, \tau) )^* \te + (\tau \circ \hat{\pr}_2)^* \om = ( [\hat{\pr}_1] \circ (id_P, \tau) )^* \te + ( [\hat{\pr}_2] \circ (id_P, \tau) )^* \om = $$
$$ = (id_P, \tau)^* ( [\hat{\pr}_1]^* \te + [\hat{\pr}_2]^* \om ).$$
Again, by the injectivity of the pullback we get the multiplicativity of the $\Gaug(P)$-action on $P$.

\

\underline{Fifth part}: recall that the right $\Gaug(P)$-action on $P$ is principal with respect to the trivial projection $P \to \{*\}$. Accordingly, denoting by $A$ the Lie algebroid of $\Gaug(P)$, we can consider the infinitesimal $\Gaug(P)$-action on $P$:
$$a: \pi^*A \arr TP, \quad a (\al)_p : = d_{1_{\pi(p)}} \hat{m}_P (\cdot, p) (\al_{\pi(p)} ).$$
Recall also that the image of $a$ coincides with the vertical bundle of $P \to \{*\}$, which in this case is simply $TP$. Together with the fact that infinitesimal free actions are injective, we see that
$$a_p: A_{\pi(p)} \arr T_p P$$
is an isomorphism for every $p \in P$. 

Therefore, we have only to show that $a_p$ sends $\g_{\pi(p)}(\om) = A_{\pi(p)} \cap \ker (\om_{1_{\pi(p)}})$ to $\ker(\te_p)$. Consider $\al \in \g_{\pi(p)}(\om)$; since the action is multiplicative
$$ \te_p (a_p (\al_x) ) = \te_p ( d_{(1_x, p)} m_P (\al_x, 0) ) = (m_P^* \te)_{(1_x, p)} (\al_x, 0) = \om_{1_x} (\al_x) + \cancel{\te_p (0)} = 0,$$
therefore $a_p (\al) \in \ker(\te_p)$. Conversely, if $a_p(\al) \in \ker(\te_p)$, for some $\al \in A_{\pi(p)}$, then $\al \in \ker (\om_{1_{\pi(p)}})$, hence $\al \in \mathfrak{g}_{\pi(p)}(\om)$.
\end{proof}

\begin{obs}\label{obs_paper_Drummond}
Proposition 3.4 from \cite{Bur13} establishes a 1-1 correspondence between connection forms $\theta$ on a principal bundle $P \to M$ and multiplicative forms $\omega$ on its gauge groupoid $\Gaug(P) \rightrightarrows M$.

Using our Proposition \ref{transportation_multiplicative_via_Morita} we can provide an alternative method to prove one implication of Bursztyn and Drummond's result. Indeed, if we consider 
as $\om_H$ the Maurer-Cartan form of $H$, by Corollary \ref{condition_multiplicativity_principal_bundle} we have that $\theta$ is a principal connection on $P$, and out of it we built a multiplicative form $\omega$ on $\Gaug(P)$. 
\end{obs}

We conclude this appendix by remarking that Proposition \ref{transportation_multiplicative_via_Morita} can be more naturally expressed in greater generality, replacing $H$ by a Lie groupoid $\mathcal{H} \rightrightarrows X$, adding a moment map $\mu: P \to X$, and replacing the isomorphism $\pi^*\g(\om) \cong \ker(\te)$ by $\pi^*\g(\om) \cong \ker(\te) \cap \ker(d\mu)$. Such setting is particularly suitable to treat the relations between multiplicative forms and Morita equivalence, which could lead to possible generalisations to the result mentioned in Remark \ref{obs_paper_Drummond}.

However, since in this paper we focus on applications to principal group bundles, we only draw below a diagram for the reader interested in a more conceptual understanding:

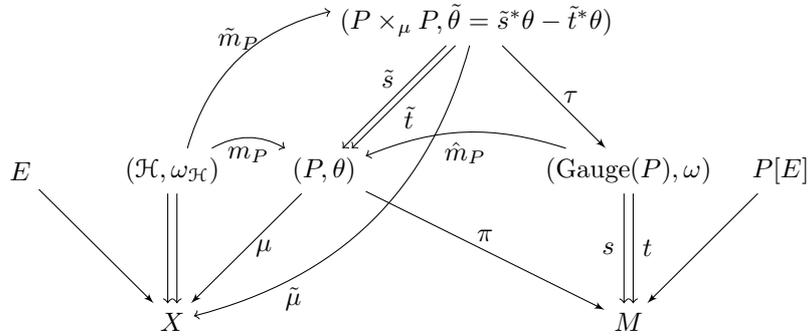
\begin{center}

  \tikzstyle{line} = [draw, -latex']
\begin{tikzpicture}[node distance = 2cm, auto]

\pgfmathsetmacro{\shift}{0.4ex}

    \node (0) {$ \quad $};
    \node [below of=0] (G) {$ (\mathcal{H}, \om_{\mathcal{H}}) $};
    \node [right of=G] (P) {$ (P, \te) $};
    \node [right of=P] (1) {$ \quad $};
    \node [above of=1] (PxP) {$ (P \times_\mu P, \tilde{\te} = \tilde{s}^* \te - \tilde{t}^* \te) $};
    \node [right of=1] (PxP/G) {$ (\Gaug(P), \om) $};
    \node [below of=G] (M) {$ X $};
    \node [below of=PxP/G] (P/G) {$ M $};

 \node [left of=G] (rrepr) {$ E $};
\node [right of=PxP/G] (repr) {$ P[E] $};

    \path [line] (P) -- (M) node [midway,right] {$ \mu $};
    \path[->] (PxP) edge [bend left=30] node[near end,below] {$ \tilde{\mu} $} (M);
    \path[->] (G) edge [bend left=30] node[midway,below] {$ m_P $} (P);
    \path[->] (G) edge [bend left=30] node[midway,above] {$ \tilde{m}_P $} (PxP);

    \path[->] (PxP/G) edge [bend right=20] node[midway,below] {$ \hat{m}_P $} (P);

    \path [line] (PxP) -- (PxP/G) node [midway,right] {$ \tau $};


\draw[->,transform canvas={xshift=-\shift}](G) to node[midway,left] {$  $}(M);
\draw[->,transform canvas={xshift=\shift}] (G) to node[midway,right]  {$  $} (M);
    

\draw[->,transform canvas={xshift=-\shift}](PxP) to node { $\tilde{t} $}(P);
\draw[->,transform canvas={xshift=\shift}] (PxP) to node[swap] {$ \tilde{s} $} (P);

\draw[->,transform canvas={xshift=-\shift}](PxP/G) to node[midway,left] {$ s $}(P/G);
\draw[->,transform canvas={xshift=\shift}] (PxP/G) to node[midway,right]  {$ t $} (P/G);

    \path [line] (P) -- (P/G) node [midway,above] {$ \pi $};
   \path [line] (rrepr) -- (M) node [midway,above] {$ $};
   \path [line] (repr) -- (P/G) node [midway,above] {$ $};

\end{tikzpicture}

\end{center}

The proof would be nearly identical, provided one considers the fibred pair groupoid $P \times_\mu P \rightrightarrows P$, the associated gauge groupoid $\Gaug(P):= (P \times_\mu P)/\mathcal{H}$, and a suitable generalisation of basic forms for principal groupoid bundles (see e.g.\ Appendix 8.8 of \cite{Yud16}).

\end{document}